\documentclass{elsarticle}
\usepackage{amsmath}
\usepackage{amsthm}
\usepackage{amssymb}
\usepackage{amscd}
\usepackage{stmaryrd}
\usepackage{amsfonts}
\usepackage{amsbsy}
\usepackage{epsfig,afterpage}
\usepackage[all]{xy}
\usepackage{graphicx}
\usepackage{color}
\usepackage{transparent}

\usepackage{import}
\usepackage[dvips]{psfrag}
\usepackage{color} 
\usepackage{enumerate}
\usepackage{overpic}
\usepackage{xcolor}
\usepackage{lineno}
\modulolinenumbers[5]
\usepackage{multicol}
\usepackage{lipsum}

\makeatletter
\DeclareFontFamily{U}{tipa}{}
\DeclareFontShape{U}{tipa}{m}{n}{<->tipa10}{}
\newcommand{\ark@char}{{\usefont{U}{tipa}{m}{n}\symbol{62}}}%

\newcommand{\ark}[1]{\mathpalette\ark@arc{$#1$}}

\newcommand{\ark@arc}[2]{%
	\sbox0{$\m@th#1#2$}%
	\vbox{
		\hbox{\resizebox{\wd0}{\height}{\ark@char}}
		\nointerlineskip
		\box0
	}%
}
\makeatother

\newtheorem {theorem} {Theorem} 

\newtheorem {definition} {Definition}

\begin{document}
\pretolerance10000

\begin{frontmatter}
	
\title{Existence of weak solutions for a degenerate Goursat type linear problem}

\author[1]{Olimpio Hiroshi Miyagaki}
\ead{olimpio@ufscar.br}
\author[2]{Carlos Alberto Reyes Peña\corref{cor1}}
\ead{carlos.reyes@estudante.ufscar.br}
\author[3]{Rodrigo da Silva Rodrigues}
\ead{rodrigorodrigues@ufscar.br}

\cortext[cor1]{Corresponding author}

\address[1]{olimpio@ufscar.br}
\address[2]{carlos.reyes@estudante.ufscar.br}
\address[3]{rodrigorodrigues@ufscar.br}




\begin{abstract}

For a generalization of the Gellerstedt operator with mixed-type Dirichlet boundary conditions to a suitable Tricomi domain, we prove the existence and uniqueness of weak solutions of the linear problem and for a generalization of this problem. The classical method introduced by Didenko, which study the energy integral argument, will be used to prove estimates for a specific Tricomi domain. 
\end{abstract}
\end{frontmatter}

\linenumbers

\section{Introduction}

Consider the following problem involving a  Gellerstedt type operator with mixed-type Dirichlet
boundary conditions
\begin{equation}\label{LPL}
\begin{cases}
  \mathcal{O}(u)-\lambda u:=-y^{m_1}u_{xx}-x^{m_2}u_{yy}-\lambda u=f(x,y)  &\mbox{in}\quad \Omega, \\
\quad \quad \quad \quad u=\gamma &\mbox{on}\quad AC\cup\sigma\subseteq\partial\Omega,
\end{cases}
\end{equation}
    where $m_1,m_2\in \mathbb{N},$ $f\in L^2(\Omega)$ and $\Omega\subset\mathbb{R}^2$ is a Tricomi domain for the operator $\mathcal{O},$ that is, $\Omega$ is open, bounded, simply connected set with piecewise $C^1$ boundary $\partial\Omega=\sigma\cup AC\cup BC$ formed by the curve $\sigma\subset\{(x,y)\in\mathbb{R}^2; \ y>0\}$ is a piecewise $C^2$ simple arc join the points $A=(-x_1,0)$ and $B=(x_0,0)$ with $0<x_0<x_1$ in $\mathbb{R}$ and the characteristics curves $AC$ and $BC$ in $\{(x,y)\in\mathbb{R}^2; \ -x_1\leq x\leq x_0 \ \mbox{and} \ y\leq0\}$ of $\mathcal{O}$ passing through points $A$ and $B$ respectively and meet at point $C$. See figure \ref{Sorv4}.

\begin{figure}[!ht]
    \centering
    \includegraphics[width=7.1cm,angle=-90]{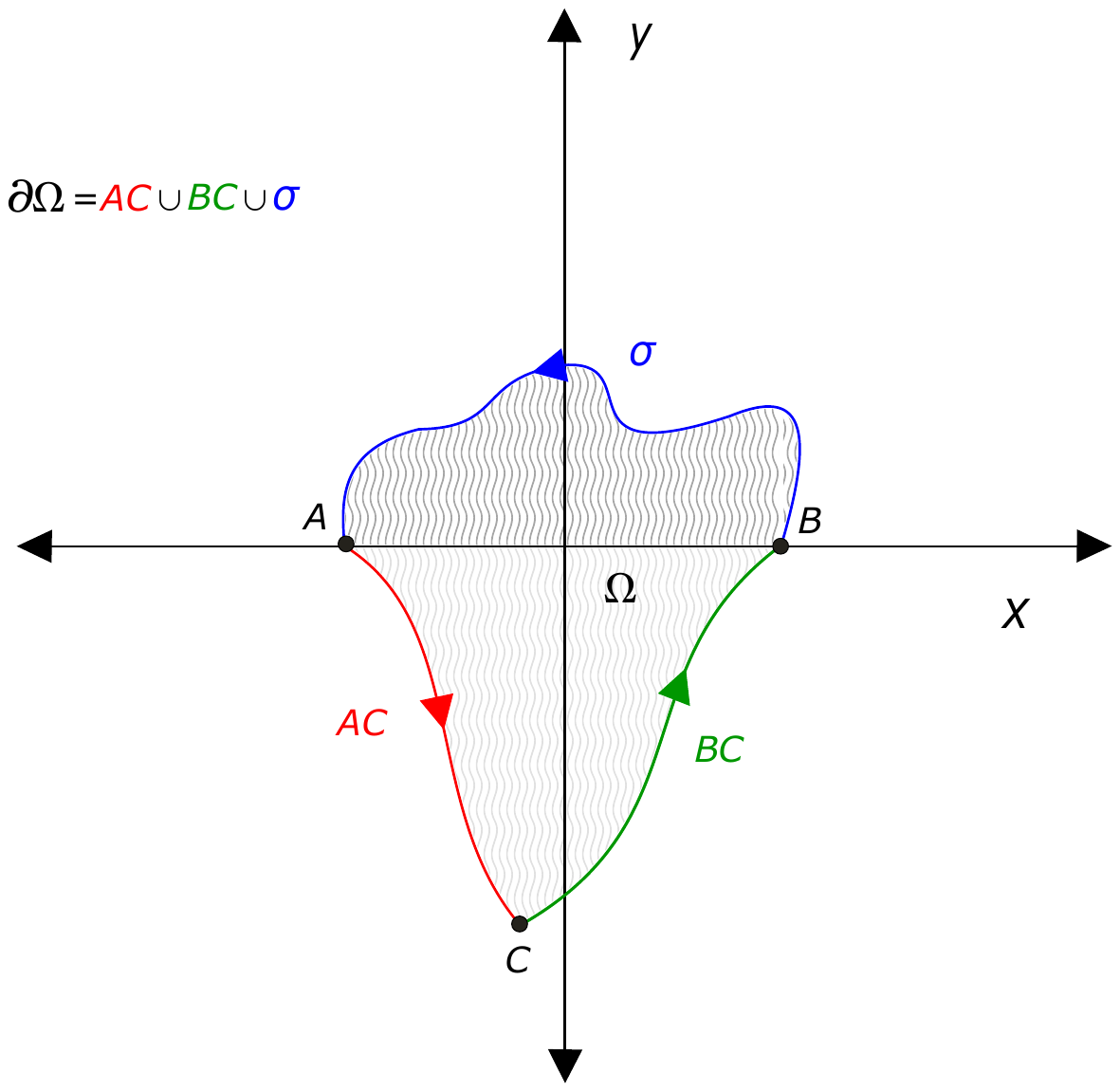}
    \caption{Domain $\Omega$ for the problem \eqref{LPL}.}
    \label{Sorv4}
\end{figure}

\vspace{0.3cm}

As developed in \cite{olirodcar}, the characteristic curves $AC$ and $BC$ in $\{(x,y)\in\mathbb{R}^2; \ y\leq0\}$ of the operator $\mathcal{O}$, for $m_1\in\mathbb{N}$ odd and $\frac{m_2}{2}\in\mathbb{N}$ even, are given by

\begin{align}\label{AC1} 
AC=\bigg\{(x,y)\in\mathbb{R}^2;& \ y_c\leq y\leq0, \nonumber\\
&\frac{2}{m_1+2}(-y)^\frac{m_1+2}{2}=\frac{2}{m_2+2}\left[x^\frac{m_2+2}{2}+x_1^\frac{m_2+2}{2}\right] \bigg\},
\end{align}

\begin{align}\label{BC1}
BC=\bigg\{ (x,y)\in\mathbb{R}^2;& \ y_c\leq y\leq0, \nonumber\\
&\frac{2}{m_1+2}(-y)^\frac{m_1+2}{2}=-\frac{2}{m_2+2}\left[x^\frac{m_2+2}{2}-x_0^\frac{m_2+2}{2}\right] \bigg\},
\end{align}
and the point $C=(x_c, y_c)$ where 

$$x_c=-\left[\frac{1}{2}\left(x_1^\frac{m_2+2}{2}-x_0^\frac{m_2+2}{2}\right)\right]^\frac{2}{m_2+2}$$
and
$$y_c=-\left[\frac{1}{2}\left(\frac{m_1+2}{m_2+2}\right)(x_1^\frac{m_2+2}{2}-x_0^\frac{m_2+2}{2})\right]^\frac{2}{m_1+2}\cdot$$

Parameterizing \eqref{AC1} and \eqref{BC1} with respect to the variable $y$, the exterior normal vector will be given by

\begin{align*}
\eta_{AC}=\left(-1,-\left[\frac{m_2+2}{m_1+2}(-y)^\frac{m_1+2}{2}-x_1^\frac{m_2+2}{2}\right]^\frac{-m_2}{m_2+2}(-y)^\frac{m_1}{2}\right),
\end{align*}
and
\begin{align*}
\eta_{BC}=\left(1,-\left[-\frac{m_2+2}{m_1+2}(-y)^\frac{m_1+2}{2}+x_0^\frac{m_2+2}{2}\right]^\frac{-m_2}{m_2+2}(-y)^\frac{m_1}{2}\right),
\end{align*}
respectively. 

The importance in the study  problem \eqref{LPL}  is due to the operator $\mathcal{O}=-y^{m_1}\partial_{xx}-x^{m_2}\partial_{yy}$ is a generalization of the operators Laplacian operator  $\Delta=\partial_{xx}+\partial_{yy}$, Tricomi operator $T=y\partial_{xx}+\partial_{yy}$ see \cite{lupopayne, Moiseev, Sabitova}, and Gellerstedt operator $L=y^{m_1}\partial_{xx}+\partial_{yy}$,  see\cite{lupopayne, Moiseev,Sabitova},  for that reason the existence results of weak solutions to the problem \eqref{LPL}, were motivated by the classical result involving the Laplacian operator in \cite{giltru} and  the linear problem involving the Tricomi operator in \cite{didenko}.  Nonexistence result of nontrivial regular solutions for the semilinear problem involving the operator $\mathcal{O}$ in \cite{olirodcar} plays an important role in our study, since some arguments made on the Tricomi domains as well as in the weighted Sobolev spaces associated with the operator were borrowed from there.
Briefly we are going  make a review of the problem. In \cite{Otway} Otway  treated the Keldysh type operator getting existence of distribution solution also some maximum principle for non nuniformly elliptic operator.  Sabitova in \cite{Sabitova} studied a eigenvalue problem involving Chaplygin type operator
$$(sign \, y) |y|^m u_{xx} +u_{yy}-\lambda |y|^m u=0$$ where uniqueness and non existence results for $\lambda$ in some range.  In \cite{Favaron,Moiseev} also eigenfunctions problem is treated.  In \cite{Paronetto}  existence for elliptic-hyperbolic Tricomi problem is established, while in \cite{Chen} is obtained estimates for Dirichlet eigenvalue for degenerate $\Delta_{\mu}$ Laplace operator. 

\vspace{0.3cm}
All previously mentioned operators, except for the Laplacian, share a common feature in terms of a weight. However, our operator involves two weights. This additional weight brings others difficulties. In fact, the gradient and the norm are always related by the divergence and this plays a crucial role in the study of PDEs. As mentioned previously in \cite{olirodcar} and not mentioned in the other works cited here, we note that, $\mathcal{O}u=div(\mathcal{X}u)\neq div(-\nabla_{m_1,m_2} u)$ and
$$||u||_{W^1_{AC\cup\sigma}}^2=||\nabla_{m_1,m_2} u||_{L^2(\Omega)}^2+||u||_{L^2(\Omega)}^2
$$ where 
$$\nabla_{m_1,m_2}u=(|y|^{\frac{m_1}{2}}u_x,|x|^{\frac{m_2}{2}}u_y) \mbox{ and } \mathcal{X}u=(-y^{m_1}u_x,-x^{m_2}u_y)$$ and the particular case, $m_1=m_2=0,$ we have $$\nabla_{m_1,m_2} u=\nabla u,\,  \mathcal{X}u=(-u_x,-u_y)=-\nabla u,\, \Delta u=div(\nabla u),$$ and $||u||_{H^1(\Omega)}^2=||\nabla u||_{L^2(\Omega)}^2+||u||_{L^2(\Omega)}^2$. Along with this, we add the non-zero lambda in \eqref{LPL} with a view to studying the spectral theory of this operator in the future.
\vspace{0.3cm}

Following Didenko's ideas in \cite{didenko}, we will study the particular linear problem where $\lambda=\gamma=0$ in \eqref{LPL} on the same domain $\Omega$.
\begin{equation}\label{LP}
\begin{cases}
  \mathcal{O}(u):=-y^{m_1}u_{xx}-x^{m_2}u_{yy}=f(x,y)   &\mbox{in}\quad \Omega, \\
\quad\quad u=0 & \mbox{on}\quad AC\cup\sigma\subseteq\partial\Omega.
\end{cases}
\end{equation} and consider also on the same domain $\Omega$ the adjoint problem to the semilinear problem \eqref{LP}
\begin{equation}\label{ALP}
\begin{cases}
  \mathcal{O}(v):=-y^{m_1}v_{xx}-x^{m_2}v_{yy}=f(x,y)  &\mbox{in}\quad \Omega, \\
\quad\quad v=0 &\mbox{on}\quad BC\cup\sigma\subseteq\partial\Omega.
\end{cases}
\end{equation}

Throughout the development of this paper, as is inevitable, we found several difficulties, mainly with respect to domain $\Omega$ which we impose some modifications to get suitable estimates. No less important and not mentioned in previous works to ours, as for the Gellerstedt operator as well for the Tricomi operators, is the interaction between the admissibility of the domain for the operator $\mathcal{O}$ and consequently for the operator $\mathcal{O}-\lambda$ with $\lambda\leq \frac{1}{K_3},$ where $K_3$ is a suitable constant. More specifically we prove:

\begin{theorem}
    Let $\Omega\subset\mathbb{R}^2$ be a admissible Tricomi domain for the operator $\mathcal{O}-\lambda$ as defined in the problem \eqref{LPL}, $f(x,y)\in L^2(\Omega)$, $\lambda\in\mathbb{R}$ and 
 $\lambda\leq0$. The problem 
    \begin{equation}\label{CREYS_1}
           \begin{cases}
         \mathcal{O}(u)-\lambda u=f(x,y)  &\mbox{in}\quad \Omega, \\
 \quad\quad\quad\quad u=\gamma &\mbox{on}\quad AC\cup\sigma\subseteq\partial\Omega, 
    \end{cases}
    \end{equation}
 admits an unique weak solution and the solution operator 
    \begin{align*}
        S^{\lambda,\gamma}_{AC\cup\sigma}: L^2(\Omega)\rightarrow  W^1_{AC\cup\sigma}
    \end{align*}
    which assigns to $f\in L^2(\Omega)$ the unique weak solution $u\in W^1_{AC\cup\sigma}$ of the problem \eqref{CREYS_1} is linear and continuous. 
\end{theorem}

This paper will be divided into two parts. In section \ref{sec2}, we define an admissible Tricomi domain, then we will prove the existence of an admissible domain for the operator $\mathcal{O}$ and consequently for the operator $\mathcal{O-\lambda}$. We present the spaces associated to the mentioned problems and their respective dual space in order to analyze how they are related and obtain a better understanding of our operator. In section \ref{sec3}, we will present the main results in this paper, we start prove uniqueness and existence of the weak solution for the problem \eqref{LP} and after with the help of this we proof the uniqueness and existence of the weak solutions for the problem \eqref{LPL}.

\section{Statement of the Results}\label{sec2}

In this section we will define spaces weighted Sobolev naturally  associated with the problems \eqref{LP} and \eqref{ALP}, their dual spaces and their norms, as well as, we will define a specific Tricomi domains for our work and we will prove the existence of this domain.

\vspace{0,3cm}

Consider, $$C_{AC\cup\sigma}^{\infty}(\overline{\Omega})=\{u\in C^\infty(\overline{\Omega}); \ \ u|_{AC\cup\sigma}=0\},$$
and
$$C_{BC\cup\sigma}^{\infty}(\overline{\Omega})=\{v\in C^\infty(\overline{\Omega}); \ \ v|_{BC\cup\sigma}=0\}.$$

Define $W^1_{AC\cup\sigma}$ and $W^1_{BC\cup\sigma}$ the closures of $C_{AC\cup\sigma}^{\infty}(\overline{\Omega})$ and $C_{BC\cup\sigma}^{\infty}(\overline{\Omega})$ with respect to the norm of $W^{1,2}(\Omega),$ respectively. By the theory of space with negative norm \cite{lax} consider the dual space $W^{-1}_{AC\cup\sigma}$ and $W^{-1}_{BC\cup\sigma}$ of $W^1_{AC\cup\sigma}$ and $W^1_{BC\cup\sigma}$ respectively with the norm 
$$||\varphi||_{W^{-1}_{AC\cup\sigma}}= \sup_{0\neq u\in W^1_{AC\cup\sigma}}\frac{|(\varphi,u)_{L^2}|}{||u||_{W^1_{AC\cup\sigma}}},$$
and
$$||\phi||_{W^{-1}_{BC\cup\sigma}}= \sup_{0\neq v\in W^1_{BC\cup\sigma}}\frac{|(\phi,v)_{L^2}|}{||v||_{W^1_{BC\cup\sigma}}}.$$

We have the following inclusions 
$$ W^1_{AC\cup\sigma}\subset L^2(\Omega)\subset W^{-1}_{AC\cup\sigma} \ \ \ \mbox{and} \ \ \ W^1_{BC\cup\sigma}\subset L^2(\Omega)\subset W^{-1}_{BC\cup\sigma}.$$

In \cite{olirodcar}, we defined $\mathcal{X}u=(-y^{m_1}u_x,-x^{m_2}u_y)$ and we observed that $div(\mathcal{X}u)=\mathcal{O}u$. By the Divergence Theorem \cite{giltru}, we have
\begin{align*}
    \int_{\Omega} div(v(-y^{m_1}u_x,-x^{m_2}u_y))=\int_{\partial\Omega} v(-y^{m_1}u_x,-x^{m_2}u_y) \cdot \eta \; ,
\end{align*}
and
\begin{align*}
    \int_{\Omega} div(u(-y^{m_1}v_x,-x^{m_2}v_y))=\int_{\partial\Omega} u(-y^{m_1}v_x,-x^{m_2}v_y) \cdot \eta \; .
\end{align*}

By the proprieties of the divergent, we obtain 
\begin{align*}
    \int_{\Omega} \nabla v(-y^{m_1}u_x,-x^{m_2}u_y)+&\int_{\Omega} v \; div(-y^{m_1}u_x,-x^{m_2}u_y)
    \\&=\int_{\partial\Omega} v(-y^{m_1}u_x,-x^{m_2}u_y) \cdot \eta \; ,
\end{align*}
and 
\begin{align*}
    \int_{\Omega} \nabla u(-y^{m_1}v_x,-x^{m_2}v_y)+&\int_{\Omega} u \; div(-y^{m_1}v_x,-x^{m_2}v_y)
    \\&
    =\int_{\partial\Omega} u(-y^{m_1}v_x,-x^{m_2}v_y) \cdot \eta \; ,
\end{align*}
respectively. Replacing, $\nabla v=(v_x,v_y)$, $\nabla u=(u_x,u_y)$ and $div(\mathcal{X}u)=\mathcal{O}u,$ we get
\begin{align*}
    \int_{\Omega} (v_x,v_y)(-y^{m_1}u_x,-x^{m_2}u_y)+\int_{\Omega} v \; \mathcal{O}u=\int_{\partial\Omega} v(-y^{m_1}u_x,-x^{m_2}u_y) \cdot \eta \; ,
\end{align*}
and
\begin{align*}
    \int_{\Omega} (u_x,u_y)(-y^{m_1}v_x,-x^{m_2}v_y)+\int_{\Omega} u \; \mathcal{O}v=\int_{\partial\Omega} u(-y^{m_1}v_x,-x^{m_2}v_y) \cdot \eta \; .
\end{align*}

Therefore, for $u\in W^1_{AC\cup\sigma}$ and $v\in W^1_{BC\cup\sigma}$, we get
\begin{align*}
    \int_{\Omega} (v_x,v_y)(-y^{m_1}u_x,-x^{m_2}u_y)+\int_{\Omega} v \; \mathcal{O}u=\int_{AC} v(-y^{m_1}u_x,-x^{m_2}u_y) \cdot \eta \; ,
\end{align*}
and
\begin{align*}
    \int_{\Omega} (u_x,u_y)(-y^{m_1}v_x,-x^{m_2}v_y)+\int_{\Omega} u \; \mathcal{O}v=\int_{BC} u(-y^{m_1}v_x,-x^{m_2}v_y) \cdot \eta \; .
\end{align*}

As made in \cite{olirodcar}, $(\mathcal{X}u\cdot\eta)_{|_{AC}}\equiv0$ and $(\mathcal{X}v\cdot\eta)_{|_{BC}}\equiv0$. Hence 

\begin{align*}
    \int_{\Omega} (v_x,v_y)(-y^{m_1}u_x,-x^{m_2}u_y)+\int_{\Omega} v \; \mathcal{O}u=0,
\end{align*}
and
\begin{align*}
    \int_{\Omega} (u_x,u_y)(-y^{m_1}v_x,-x^{m_2}v_y)+\int_{\Omega} u \; \mathcal{O}v=0.
\end{align*}

In conclusion, $$(\mathcal{O}u,v)_{L^2}=\int_{\Omega} v \; \mathcal{O}u=\int_{\Omega} (y^{m_1}u_xv_x+x^{m_2}u_yv_y)=\int_{\Omega} u \; \mathcal{O}v=(u,\mathcal{O}v)_{L^2}.$$


Note that, for $u\in W^1_{AC\cup\sigma}$ 
\begin{align*}
||\mathcal{O}u||_{W^{-1}_{BC\cup\sigma}}=& \sup_{0\neq v\in W^1_{BC\cup\sigma}}\frac{|(\mathcal{O}u,v)_{L^2}|}{||v||_{W^1_{BC\cup\sigma}}} \\
=& \sup_{0\neq v\in W^1_{BC\cup\sigma}}\frac{|\int_\Omega(y^{m_1}u_xv_x+x^{m_2}u_yv_y)|}{||v||_{W^1_{BC\cup\sigma}}} \\
=& \sup_{0\neq v\in W^1_{BC\cup\sigma}}\frac{|(u,\mathcal{O}v)_{L^2}|}{||v||_{W^1_{BC\cup\sigma}}} \\
\leq& \sup_{0\neq v\in W^1_{BC\cup\sigma}}\frac{|(u,\mathcal{O}v)_{L^2}|}{K_1||v||_{L^2}} \\
\leq& \sup_{0\neq v\in W^1_{BC\cup\sigma}}\frac{||u||_{L^2}||\mathcal{O}v||_{L^2}}{K_1||v||_{L^2}},
\end{align*}
since $\mathcal{O}$ is continue in $L^2$. Then, 
\begin{align*}
||\mathcal{O}_{AC}u||_{W^{-1}_{BC\cup\sigma}}\leq C_1 ||u||_{W^1_{AC\cup\sigma}} \ , \ \ u\in W^1_{AC\cup\sigma},
\end{align*}
similarly, for $v\in W^1_{BC\cup\sigma}$
\begin{align*}
||\mathcal{O}_{BC}v||_{W^{-1}_{AC\cup\sigma}}\leq C_2 ||v||_{W^1_{BC\cup\sigma}} \ , \ \ v\in W^1_{BC\cup\sigma},
\end{align*}
where $\mathcal{O}_{AC}u$ and $\mathcal{O}_{BC}v$ is the unique continuous extensions of $\mathcal{O}$ relative to the dense subspaces $C_{AC\cup\sigma}^{\infty}(\overline{\Omega})$ and $C_{BC\cup\sigma}^{\infty}(\overline{\Omega})$ respectively. The following definitions, propositions and remarks are adaptations of the definitions found in \cite{lupopayne} for our purpose.


\begin{definition}
Let $\lambda\in\mathbb{R}$. A Tricomi domain $\Omega$ is said \textcolor{blue}{admissible} for the operator $\mathcal{O}-\lambda$, if there are positive constants  $C_3$ and $C_4$ such that
\begin{align}\label{est3}
||u||_{L^2(\Omega)}\leq C_3||(\mathcal{O}_{AC} -\lambda)u||_{W^{-1}_{BC\cup\sigma}}, \ \ u\in W^1_{AC\cup\sigma},
\end{align}
and
\begin{align}\label{est4}
||v||_{L^2(\Omega)}\leq C_4||(\mathcal{O}_{BC} -\lambda)v||_{W^{-1}_{AC\cup\sigma}}, \ \ v\in W^1_{BC\cup\sigma}.
\end{align}
\end{definition}

\begin{theorem}
    Let $\Omega$ a Tricomi domain with $A=(-x_1,0)$ and $B=(x_0,0)$ with $0<x_0<x_1$ suppose that $\Omega$ and its boundary $\partial\Omega=\sigma\cup AC\cup BC$ satisfies:
    \begin{enumerate}
        \item[\textit{i)}] $|x|\leq x_1$ on $\overline{\Omega}$,
        \item[\textit{ii)}] The curve $\sigma$ is given by a $C^2$ graph $y=g(x)$ for $-x_1\leq x \leq x_0$ with $g(-x_1)=0=g(x_0)$,
        \item[\textit{iii)}] There is a constant $$h<{\left(\frac{x^{m_2}_c}{y^{m_1}_c}\right)}^{\frac{1}{2}}$$ such that $$-h<g'(x)<h \ \ \ \mbox{for} \ \ \ -x_1<x<x_0.$$
    \end{enumerate}
Then $\Omega$ is admissible for the operator $\mathcal{O}-\lambda$ for all $\lambda \leq \frac{1}{K_3},$ for a suitable $K_3>0.$ 
\end{theorem}

\begin{proof}

Initially, we will prove for case $\lambda=0.$ The proof of this case will be done in six steps. In the  step 1, we will define an auxiliary Cauchy problem. While in step 2, we prove the existence and uniqueness of the solution for the previous problem. In step 3 and step 4, we will estimate $\int_{AC} (y^{m_1}v^2_x+x^{m_2}v^2_y)(a,b)\eta_{AC}$ and $\int_\Omega v \,\mathcal{O}u$, respectively. We will relate the step 3 and the step 4 in step 5. Finally, we will estimate $(Iu,\mathcal{O}u)_{L^2}$ from above and from bellow in the step 6 to prove the inequality \eqref{est3}, where $I$ is the operator solution of Cauchy problem.

\textbf{Step 1.} Consider, $a=-(1+\epsilon x)$, $b=-h(1+\epsilon x)$ and $c=0$ where $h$ is given by the hypotheses $iii)$ and $0<\epsilon<\frac{1}{x_0}$. This way, $a$ and $b$ never vanish on $\overline{\Omega}$. Define the auxiliary Cauchy problem as being
\begin{align}\label{PCA}
    \begin{cases}
        Dv=av_x+bv_y=u  &\mbox{in} \quad \Omega, \\
        \quad v=0  &\mbox{on} \quad BC\cup\sigma\subseteq\partial\Omega.
    \end{cases}
\end{align}

\textbf{Step 2.} We will use the following change of coordinates 
$$(\xi,\eta)=\Phi(x,y)=(y+h(x+x_1),-y+h(x+x_1)).$$

Define, $w(\xi,\eta)=v(x,y)$ and note that,
\begin{align*}
    Dw&=aw_\xi h+aw_\eta h+bw_\xi-bw_\eta \\
    &=(ah+b)w_\xi+(ah-b)w_\eta,
\end{align*}
as $ah=b$ then, 
\begin{align*}
    Dw=2ahw_\xi.
\end{align*}

See that, 
$$\Phi(A)=(0,0), \, \Phi(B)=(h(x_0+x_1),h(x_0+x_1)$$
and
$$\Phi(C)=(y_c+h(x_c+x_1),-y_c+h(x_c+x_1).$$

Since the coordinate lines $\{\eta_0 = Cte \}$ intersect $\Phi(BC\cup\sigma)$ only once, we can conclude that $\Phi(BC\cup\sigma)$ can be written as the graph of a function $\psi(\eta)=\xi$ for $\eta\in [0, -y_c+h(x_c+x_1)]$. See figure \ref{mudevar}.

\begin{figure}[!ht]
    \centering
    \includegraphics[trim = 40mm 0mm 30mm 0mm, clip, width=5cm,angle=90]{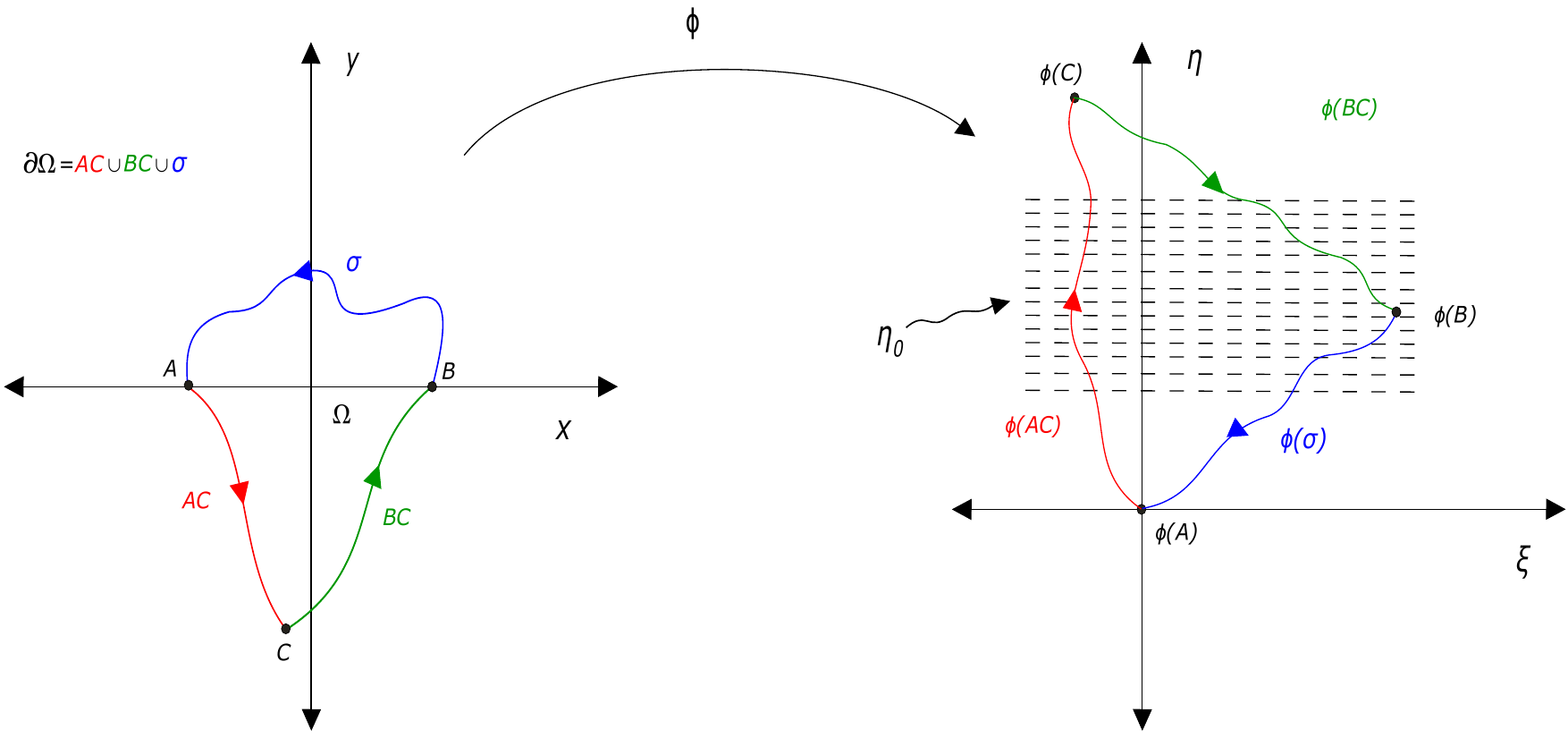}
    \caption{Change of coordinates $\Phi:\Omega \rightarrow \Phi(\Omega)$.}
    \label{mudevar}
\end{figure}

Integrating on the variable $\xi$ we get,
\begin{align*}
    v(x,y)=w(\xi,\eta)=\frac{1}{2h}\int^\xi_{\psi(\eta)} (a^{-1}u)(\Phi^{-1}(t,\eta))dt.
\end{align*}

Define $Iu=v$ as been the operator solution such that for each $u\in W^1_{AC\cup\sigma}$ associate the unique solution $v\in W^1_{BC\cup\sigma}$ of the auxiliary Cauchy problem \eqref{PCA}.

\vspace{0,3cm}

\textbf{Step 3.} We will find a suitable expression to $$\int_{AC}(y^{m_1}v_x^2+x^{m_2}v_y^2)(a,b)\eta_{AC}.$$

By the divergence Theorem in \cite{giltru}, we get 
\begin{align*}
    \int_\Omega div \big( (y^{m_1}v_x^2+x^{m_2}v_y^2)(a,b)\big)=\int_{\partial\Omega}(y^{m_1}v_x^2+x^{m_2}v_y^2)(a,b)\eta_{AC},
\end{align*}

by the proprieties of the divergent 
\begin{align*}
    \int_\Omega \nabla(y^{m_1}v_x^2+x^{m_2}v_y^2) (a,b)+ \int_\Omega  (y^{m_1}v_x^2&+x^{m_2}v_y^2) \, div (a,b) \\
    &=\int_{\partial\Omega}(y^{m_1}v_x^2+x^{m_2}v_y^2)(a,b)\eta_{AC}.
\end{align*}

Developing we have,
\begin{align*}
    \int_\Omega \nabla(y^{m_1}v_x^2+x^{m_2}&v_y^2) (a,b)
    =\int_\Omega (2y^{m_1}v_xv_{xx}+m_2^{m_2-1}v_y^2 \\
    &+2x^{m_2}v_yv_{xy},m_1y^{m_1-1}v_x^2+2y^{m_1}v_xv_{xy}+2x^{m_2}v_yv_{yy})(a,b) \\
    =&\int_\Omega (2ay^{m_1}v_xv_{xx}+am_2^{m_2-1}v_y^2 \\
    &+2ax^{m_2}v_yv_{xy}+bm_1y^{m_1-1}v_x^2+2by^{m_1}v_xv_{xy}+2bx^{m_2}v_yv_{yy}),
\end{align*}

and

\begin{align*}
    \int_\Omega  (y^{m_1}v_x^2+x^{m_2}v_y^2) \, div (a,b)=&\int_\Omega (y^{m_1}v_x^2+x^{m_2}v_y^2)(a_x+b_y) \\ =&\int_\Omega (a_x y^{m_1}v_x^2+a_x x^{m_2}v_y^2+b_y y^{m_1}v_x^2+b_y x^{m_2}v_y^2).
\end{align*}

We know that $v\equiv 0$ on $BC\cup\sigma$ then $y^{m_1}v_x^2+x^{m_2}v_y^2\equiv 0$ on $BC\cup\sigma$ as done in Step 1 in \cite{olirodcar}. So, 
\begin{align}\label{Passo3}
    \int_{AC}(y^{m_1}v_x^2+x^{m_2}v_y^2&)(a,b)\eta_{AC}=\int_\Omega (2ay^{m_1}v_xv_{xx}+am_2^{m_2-1}v_y^2 \nonumber \\
    &+2ax^{m_2}v_yv_{xy}+bm_1y^{m_1-1}v_x^2+2by^{m_1}v_xv_{xy} +2bx^{m_2}v_yv_{yy} \nonumber \\ &+a_x y^{m_1}v_x^2+a_x x^{m_2}v_y^2+b_y y^{m_1}v_x^2+b_y x^{m_2}v_y^2).
\end{align}

\textbf{Step 4.} We will find a suitable expression to $$\int_\Omega v \,\mathcal{O}u.$$

We know that,
\begin{align}\label{vOu}
    \int_{\Omega} v \,\mathcal{O}u=\int_\Omega (y^{m_1}u_xv_y+x^{m_2}u_xv_y).
\end{align}

As $Dv=av_x+bv_y=u$ on $\Omega$ then,
\begin{align}\label{uxuy}
    u_x=a_xv_x+av_{xx}+b_xv_y+bv_{xy} \;\;\mbox{and}\;\; u_y=a_yv_x+av_{xy}+b_yv_y+bv_{yy}.
\end{align}

Replacing \eqref{uxuy} in \eqref{vOu}, we have

\begin{align}\label{passo4}
    \int_{\Omega} v \,\mathcal{O}u=&\int_\Omega (a_xy^{m_1}v^2_x+ay^{m_1}v_xv_{xx}+b_xy^{m_1}v_xv_y+by^{m_1}v_xv_{xy} \nonumber \\  &+a_yx^{m_2}v_xv_y+ax^{m_2}v_yv_{xy}+b_yx^{m_2}v^2_y+bx^{m_2}v_yv_{yy})
\end{align}

\vspace{.3cm}
\textbf{Step 5.} We will find suitable coefficients $\alpha$, $\beta$ and $\gamma$ such that 
$$\int_{\Omega} v \,\mathcal{O}u = \int_{\Omega} \alpha v^2_x+2\beta v_xv_y+\gamma v^2_y+\frac{1}{2}\int_{AC}(y^{m_1}v^2_x+x^{m_2}v^2_y)(a,b)\eta_{AC}.$$

We know that $Dv=av_x+bv_y=u$. On $AC$ we have $u\equiv 0$ then $av_x=-bv_y=-ahv_y$. So, $h^2v^2_y=v^2_x$. Therefore, 
\begin{align*}
    (y^{m_1}v^2_x+x^{m_2}v^2_y)(a,b)\eta_{AC}=(y^{m_1}h^2+x^{m_2})v^2_y(-(1+\epsilon x), -h(1+\epsilon x))\eta_{AC}.
\end{align*}
 How $-x_1<x<x_c\leq 0$ and $y_c\leq y\leq 0$ on $AC$ hence $x^{m_2}_c\leq x^{m_2}$ and $y^{m_1}_c\leq y^{m_1}$ because $m_1\in\mathbb{N}$ and $\frac{m_2}{2}\in\mathbb{N}$ are odd and even numbers, respectively. By using the hypothesis $h<{\left(\dfrac{x^{m_2}_c}{y^{m_1}_c}\right)}^{\frac{1}{2}},$ we get 
\begin{align*}
    h<\left(\frac{x^{m_2}_c}{-y^{m_1}_c}\right)^\frac{1}{2}\leq \left(\frac{x^{m_2}}{-y^{m_1}}\right)^\frac{1}{2}.
\end{align*}
Then, $y^{m_1}h^2+x^{m_2}>0$. In conclusion, 
$$\int_{AC}(y^{m_1}v^2_x+x^{m_2}v^2_y)(a,b)\eta_{AC}\geq 0.$$ 

Combining \eqref{Passo3} and \eqref{passo4} and the last inequality, we obtain
$$\int_{\Omega} v \,\mathcal{O}u \geq \int_{\Omega} \alpha v^2_x+2\beta v_xv_y+\gamma v^2_y,$$
where $\alpha=-\epsilon y^{m_1}-h(1+\epsilon x)m_1y^{m_1-1}$, $\beta=-h\epsilon y^{m_1}$ and $\gamma=(1+\epsilon x)m_2x^{m_2-1}+\epsilon x^{m_2}$. By the hypotheses \textit{ii).} and \textit{iii)} and the choice of $0<\epsilon<\frac{1}{x_0}$ we can proof with the standard calculations that $\alpha\gamma-\beta^2>0$.

We can write 
\begin{align*}
    \mathbf{A}=\left(\begin{array}{cc}
\alpha & \beta\\
\beta & \gamma
\end{array}\right)
\end{align*}
and see that 

\begin{align*}
    \left(\begin{array}{cc}
     v_x    &  v_y 
    \end{array}\right)
    \left(\begin{array}{cc}
\alpha & \beta\\
\beta & \gamma
\end{array}\right)
\left(\begin{array}{cc}
 v_x\\
 v_y
\end{array}\right)
= \alpha v^2_x+2\beta v_xv_y+\gamma v^2_y.
\end{align*}
  On the matrix space $\mathbb{M}_{2x1}$, we can consider the norm $|| X ||_1=X A X ^T$ para $A$ defined positively, because $\alpha.\beta - \gamma^2 >0$, and we have the stard norm on $\mathbb{M}_{2x1}$ given by $||X||_2=X X^T$. As $\mathbb{M}_{2x1}$ is a finite dimensional space we know that exist $\delta>0$ such that $\delta||X||_2\leq||X||_1$ for every matrix $X\in\mathbb{M}_{2x1}$. Consider $X=\left(\begin{array}{cc} v_x & v_y \end{array}\right)$ we have,

$$\int_{\Omega} v \,\mathcal{O}u \geq \int_{\Omega} \alpha v^2_x+2\beta v_xv_y+\gamma v^2_y.\geq \delta\int_\Omega v^2_x+v^2_y.$$
This is, 
\begin{align*}
    \int_{\Omega} v \,\mathcal{O}u \geq \delta ||\nabla v||^2_{L^2(\Omega)}.
\end{align*}

\textbf{Step 6.} Note that, 
\begin{align*}
    (Iu,\mathcal{O}u)_{L^2(\Omega)}=\int_\Omega v \,\mathcal{O}u.
\end{align*}
So, $(Iu,\mathcal{O}u)_{L^2(\Omega)}\geq \delta ||\nabla v||^2_{L^2(\Omega)}.$ By the generalize Cauchy inequality we get,
\begin{align*}
    \delta ||\nabla v||^2_{L^2(\Omega)}\leq(Iu,\mathcal{O}u)_{L^2(\Omega)}\leq ||Iu||_{W^1_{BC\cup\sigma}}||\mathcal{O}u||_{W^{-1}_{BC\cup\sigma}}
\end{align*}

On the other hand, by definition for any $v\in W^1_{BC\cup\sigma}$ we have
\begin{align*}
    ||v||^2_{W^1_{BC\cup\sigma}}=||\nabla_{m_1,m_2} v||^2_{L^2(\Omega)}+||v||^2_{L^2(\Omega)}.
\end{align*}

By the Poincare inequality, exist $k>0$ such that 
\begin{align*}
    ||v||_{L^2{\Omega)}}\leq k ||\nabla v||_{L^2(\Omega)}.
\end{align*}

    See that, $||\nabla_{m_1,m_2} v||_{L^2(\Omega)}\leq \overline{k}||\nabla v||_{L^2(\Omega)}$. In fact, as $\Omega$ is a bounded set, there is a constant $K\in\mathbb{R}$ such that $|y|^{\frac{m_1}{2}}<K^{m_1}$ and $|x|^{\frac{m_2}{2}}<K^{m_2}$. Therefore,
\begin{align*}
    ||v||^2_{W^1_{BC\cup\sigma}}\leq \overline{k}^2||\nabla v||^2_{L^2(\Omega)}+k^2||\nabla v||^2_{L^2(\Omega)}\leq {k'}^2 ||\nabla v||^2_{L^2(\Omega)}.
\end{align*}

Since, $Iu=v$ in $\Omega$ we get,
\begin{align*}
   C||Iu||^2_{W^1_{BC\cup\sigma}} \leq||Iu||_{W^1_{BC\cup\sigma}}||\mathcal{O}u||_{W^{-1}_{BC\cup\sigma}},
\end{align*}
so, 
\begin{align*}
     C||Iu||_{W^1_{BC\cup\sigma}} \leq ||\mathcal{O}u||_{W^{-1}_{BC\cup\sigma}}.
\end{align*}

Remembering that $D$ is a first order differential operator with smooth coefficients in $\Omega$ we know that there is a constant $0<C_2 \in\mathbb{R}$ such that 
\begin{align*}
    ||Dv||_{L^2(\Omega)} \leq C_2||v||_{W^1_{BC\cup\sigma}}
\end{align*}
then, 
\begin{align*}
    ||u||_{L^2(\Omega)}=||Dv||_{L^2(\Omega)}\leq C_2||v||_{W^1_{BC\cup\sigma}}\leq \frac{C_2}{C}||Iu||_{W^1_{BC\cup\sigma}}\leq \frac{C_2}{C}||\mathcal{O}u||_{W^{-1}_{BC\cup\sigma}}.
\end{align*}
This proof the estimate \eqref{est3} $(\lambda=0)$. To proof the estimate \eqref{est4} $(\lambda=0)$ is the similar to the previous proof considering $a=(1-\epsilon x)$, $b=-ah$ and $c=0$ using the inequality $-h<g'(x)$ of the hypothesis \textit{iii)}.

Now, we will prove for the general case $\lambda\leq \frac{1}{K_3}.$ Since $\Omega$ is admissible for the operator $\mathcal{O}$ we win the propriety that $\Omega$ is admissible for the operator $\mathcal{O-\lambda}$ for a suitable $\lambda$. Indeed, 
$$||u||_{L^2(\Omega)}\leq C_3||\mathcal{O}_{AC} u-\lambda u+\lambda u||_{W^{-1}_{BC\cup\sigma}}\cdot$$
Then, 
$$||u||_{L^2(\Omega)}\leq C_3||\mathcal{O}_{AC} u-\lambda u||_{W^{-1}_{BC\cup\sigma}}+||\lambda u||_{W^{-1}_{BC\cup\sigma}}\cdot$$
There is a constant $K_3>0$ such that $||\lambda u||_{W^{-1}_{BC\cup\sigma}}\leq K_3||\lambda u||_{L^2(\Omega)}$. So, 
$$||u||_{L^2(\Omega)}-K_3||\lambda u||_{L^2(\Omega)}\leq C_3||\mathcal{O}_{AC} u-\lambda u||_{W^{-1}_{BC\cup\sigma}}\cdot$$
Therefore, for $\lambda\leq\frac{1}{K_3}$ we consider $\overline{C}_3=\frac{C_3}{1-\lambda K_3}$ and we get 
$$||u||_{L^2(\Omega)}\leq \overline{C}_3||\mathcal{O}_{AC} u-\lambda u||_{W^{-1}_{BC\cup\sigma}}\cdot$$
In conclusion, we proof the estimate \eqref{est3}, the proof of estimate \eqref{est4} is similar to the previous one.
\end{proof}

\begin{definition}\label{Def4}
     Give $f\in L^2(\Omega)$ we said that $u\in W^1_{AC\cup\sigma}$ is a \textcolor{blue}{weak solution} of \, \eqref{LPL} if the following relation is hold 
\begin{align*}
      \mathcal{B}_\lambda(x,y):= \int_\Omega (y^{m_1}u_xv_x+x^{m_2}u_yv_y-\lambda uv) \ dxdy= \int_{\Omega} fv \ dxdy, 
\end{align*}
for all $v\in C_{BC\cup\sigma}^{\infty}(\overline{\Omega}).$
\end{definition}


 \section{Existence Results}\label{sec3}

The existence results presented in this Section depend on the Riesz representation theorem for space with negative norm \cite{lax} and the others traditional result of the functional analyse whole proofs can be found in \cite{brezis}.

To prove our main theorem we need to guarantee the existence of a solution for Problem (\ref{LPL}) with $\lambda =0$ and $\gamma =0;$ what we will do next.

\begin{theorem}\label{TE1}
    Let $\Omega$ be an admissible Tricomi domain for the operator $\mathcal{O}$. For every $f\in L^2(\Omega)$ there exists an unique weak solution $u\in W^1_{AC\cup\sigma}$ of the problem \eqref{LP}. Moreover, the solution operator
    \begin{align*}
        S_{AC\cup\sigma}: L^2(\Omega)\rightarrow W^1_{AC\cup\sigma},
    \end{align*}
    which assigns to $f\in L^2(\Omega)$ the unique weak solution $u\in W^1_{AC\cup\sigma}$ of the problem \eqref{LP} is linear and continuous. 
    \end{theorem}

    \begin{proof}
\textbf{Uniqueness:} Let $u_1\in W^1_{AC\cup\sigma}$ and $u_2\in W^1_{AC\cup\sigma}$ two weak solutions of the problem \eqref{LP}. Suppose that $u_1\neq u_2$. By the definition, $<\mathcal{O}_{AC}u_1,v>_{AC}=(f,v)_{L^2}$ and $<\mathcal{O}_{AC}u_2,v>_{AC}=(f,v)_{L^2}$. So, $<\mathcal{O}_{AC}u_1-\mathcal{O}_{AC}u_2,v>_{AC}=0$. Thus, $\mathcal{O}_{AC}(u_1-u_2)=0$. Using \eqref{est3}, with $\lambda=0,$ gives $||u_1-u_2||=0$ and this is a contradiction with the suppose.

\vspace{0.3cm}

\textbf{Existence:} Define the linear functional
\begin{align*}
    \mathcal{L}_f: \mathcal{O}(C_{BC\cup\sigma}^{\infty}(\overline{\Omega}))\rightarrow \mathbb{R},
\end{align*}
as being $\mathcal{L}_f(\mathcal{O}v)=(v,f)_{L^2}$. Extending by continuity to $\mathcal{O}(W^1_{BC\cup\sigma})$ gives
\begin{align*}
    \mathcal{L}'_f: \mathcal{O}(W^1_{BC\cup\sigma})\rightarrow \mathbb{R},
\end{align*}
as being $\mathcal{L}'_f(\mathcal{O}v)=(v,f)_{L^2}$.

\vspace{0.3cm}

We know that, $\mathcal{O}(W^1_{BC\cup\sigma})\subset W^{-1}_{AC\cup\sigma}$ is a vector subspace. By the Hahn-Banach Theorem exist $\overline{\mathcal{L}}_f$ a extension of $\mathcal{L}'_f$. This is, $\overline{\mathcal{L}}_f|_{\mathcal{O}(W^1_{BC\cup\sigma})}=\mathcal{L}'_f$ and particularity $\overline{\mathcal{L}}_f|_{\mathcal{O}(C_{BC\cup\sigma}^{\infty}(\overline{\Omega}))}=\mathcal{L}_f$. Applying the Riesz Representation Theorem for space with negative norm \cite{lax}. Exist $u\in W^1_{AC\cup\sigma}$ such that 
\begin{align*}
    \overline{\mathcal{L}}_f(b)=<u,b>, \ \ \forall \ b\in W^{-1}_{AC\cup\sigma}
\end{align*}
particularity, for all $b\in \mathcal{O}(W^1_{BC\cup\sigma})\subset W^{-1}_{AC\cup\sigma}$. Hence for $b=\mathcal{O}v$ gives
\begin{align*}
    \mathcal{L}'_f(\mathcal{O}v)=(v,f)_{L^2}=<u,\mathcal{O}v>_{AC}.
\end{align*} 
This is equivalent to the condition \textit{ii)} of the definition \ref{Def4}, with $\lambda=0,$ as mentioned early. 
\end{proof}

In the next, we will proof the main theorem. 

\begin{theorem}
    Let $\Omega$ be an admissible Tricomi domain for the operator $\mathcal{O}-\lambda$, $f\in L^2(\Omega)$, $\gamma\in\mathbb{R}$ and $\lambda\leq0$. The problem \eqref{LPL} admits an unique weak solution in the sense of the previous definition and the solution operator 
    \begin{align*}
S^{\lambda,\gamma}_{AC\cup\sigma}: L^2(\Omega)\rightarrow W^1_{AC\cup\sigma}
    \end{align*}
    which assigns to $f\in L^2(\Omega)$ the unique weak solution $u\in W^1_{AC\cup\sigma}$ of the problem \eqref{LPL} is linear and continuous. 
\end{theorem}

\begin{proof}
    The proof will be done in three cases analyzing the similarity with the Theorem \ref{TE1}. 
    
\vspace{0.3cm}
\textbf{Case 1.} Consider $\gamma=0$ and $\lambda=0$. Here, we have the Theorem \ref{TE1}.

\vspace{0.3cm}    
\textbf{Case 2.} Consider $\gamma=0$ and $\lambda<0$. \textbf{Uniqueness:} Note that, $\Omega$ is an admissible Tricomi domain for the operator $\mathcal{O}-\lambda$. Let $u_1\in W^1_{AC\cup\sigma}$ and $u_2\in W^1_{AC\cup\sigma}$ two weak solutions of the problem \eqref{LPL}. Suppose that $u_1\neq u_2$. By definition $(\mathcal{O}_{AC}u_1-\lambda u_1,v)_{L^2}=(f,v)_{L^2}$ and $(\mathcal{O}_{AC}u_2-\lambda u_2,v)_{L^2}=(f,v)_{L^2}$ for all $v\in C_{BC\cup\sigma}^{\infty}(\overline{\Omega})$. So, $(\mathcal{O}_{AC}u_1-\mathcal{O}_{AC}u_2-\lambda u_1+\lambda u_2,v)_{L^2}=0$ for all $v\in C_{BC\cup\sigma}^{\infty}(\overline{\Omega})$. Thus, $\mathcal{O}_{AC}(u_1-u_2)-\lambda (u_1-u_2)=0$.

\vspace{0.3cm}

By \eqref{est3}, we have
\begin{align*}
||u_1-u_2||_{L^2}&\leq C_3||\mathcal{O}_{AC}(u_1-u_2)-\lambda(u_1-u_2)||_{W^{-1}_{AC\cup\sigma}}=0.
\end{align*}

Concluding the uniqueness of the weak solution.

\vspace{0.3cm}
\textbf{Existence:} The proof is analogous to the proof of Theorem \eqref{TE1} taking $\mathcal{O}-\lambda I$ in place of $\mathcal{O}$. 


\vspace{0.3cm}
\textbf{Case 3.} Consider $\gamma\neq0$ and $\lambda\leq0$. Here, take $u=w+\gamma$ where $w\in W^1_{AC\cup\sigma}$ is a weak solution of the next problem

\begin{equation*}
\begin{cases}
  \mathcal{O}(w)-\lambda w=f(x,y)+\lambda\gamma   &\mbox{in}\quad \Omega, \\
 \quad\quad \quad \quad w=0 &\mbox{on}\quad AC\cup\sigma\subseteq\partial\Omega,
\end{cases}
\end{equation*}
In fact, $g=f+\lambda\gamma\in L^2$ and $\lambda\leq0$. By the previous case, there is an unique weak solution $w\in W^1_{AC\cup\sigma}$ for this problem. Replacing, $w=u-\lambda$ in the anterior problem we get the result. 
\end{proof}




\section*{Acknowledgements}
C. Reyes Peña is doctoral student in the Graduate Program in Mathematics at Federal University of S\~{a}o Carlos (PPGM-UFSCar) (2019-2024). This study was financed in part by the Coordenação de Aperfeiçoamento de Pessoal de Nível Superior (CAPES) Finance Code 001. O H Miyagaki was supported in part by CNPq Nº 303256/2022-2 and FAPESP Nº 2020/16407-1.


\nolinenumbers


\end{document}